\documentclass[submission]{FPSAC2018}

\usepackage{color}
\usepackage{graphicx}
\usepackage{colortbl}
\usepackage{mathtools}
\usepackage{float}
\usepackage{lipsum}

\newtheorem{theorem}{Theorem}[section]

\newtheorem{lemma}[theorem]{Lemma}

\newtheorem{example}{Example}
\newtheorem{proposition}[theorem]{Proposition}
\newtheorem{corollary}[theorem]{Corollary}

\newtheorem{conjecture}{Conjecture}

\newenvironment{proofsketch}{%
  \proof}{\endproof}

\title{On the Schur Expansion of Jack Polynomials}

\author{Per Alexandersson\thanks{\href{mailto:per.w.alexandersson@gmail.com}{per.w.alexandersson@gmail.com}. Per Alexandersson was partially supported by the Knut and Alice Wallenberg Foundation (2013.03.07)}\addressmark{1}, Jim Haglund\thanks{\href{mailto:jhaglund@math.upenn.edu}{jhaglund@math.upenn.edu}. James Haglund was partially supported by NSF grant DMS-1600670}\addressmark{2}\and George Wang \thanks{\href{mailto:wage@math.upenn.edu}{wage@math.upenn.edu}. George Wang was partially supported by NSF grant DGE-1321851}\addressmark{2}}

\address{\addressmark{1}Department of Mathematics, KTH Royal Institute of Technology \\ \addressmark{2}Department of Mathematics, University of Pennsylvania}

\newlength\cellsize 

\savebox2{%
\begin{picture}(15,15)
\put(0,0){\line(1,0){15}}
\put(0,0){\line(0,1){15}}
\put(15,0){\line(0,1){15}}
\put(0,15){\line(1,0){15}}
\end{picture}}

\savebox3{%
\begin{picture}(15,15)
\put(0,0){\line(1,0){17}}
\put(0,15){\line(1,0){17}}
\end{picture}}

\newcommand\cellify[1]{\def\thearg{#1}\def\nothing{}%
\ifx\thearg\nothing\vrule width0pt height\cellsize depth0pt%
  \else\hbox to 0pt{\usebox2\hss}\fi%
  \vbox to 15\unitlength{\vss\hbox to 15\unitlength{\hss$#1$\hss}\vss}}

\newcommand\tableau[1]{\vtop{\let\\=\cr
\setlength\baselineskip{-12000pt}
\setlength\lineskiplimit{12000pt}
\setlength\lineskip{0pt}
\halign{&\cellify{##}\cr#1\crcr}}}

\newcommand{\graybox}{\textcolor[RGB]{165,165,165}{\rule{1\cellsize}{1\cellsize}}\hspace{-\cellsize}\usebox2}

\newcommand{\grayboxtwo}{\textcolor[RGB]{220,220,220}{\rule{1\cellsize}{1\cellsize}}\hspace{-\cellsize}\usebox2}

\received{\today}

\abstract{We present positivity conjectures for the Schur expansion of Jack symmetric functions in two bases given by binomial coefficients. Partial results suggest that there are rich combinatorics to be found in these bases, including Eulerian numbers, Stirling numbers, quasi-Yamanouchi tableaux, and rook boards. These results also lead to further conjectures about the fundamental quasisymmetric expansions of these bases, which we prove for special cases.}


\usepackage[backend=bibtex]{biblatex}
\begin{document}

\maketitle

\section{Introduction}
The (integral form, type $A$) Jack polynomials $J_\mu ^{(\alpha)} (X)$ are an important family of 
symmetric functions with applications to many areas, including statistics, mathematical physics, 
representation theory, and algebraic combinatorics. They depend on a set of variables $X$ and a 
parameter $\alpha$, and specialize into several other families of symmetric polynomials: monomial 
symmetric functions $m_\mu$ $(\alpha = \infty)$,  elementary symmetric functions 
$e_{\mu'}$ $(\alpha = 0)$, Schur functions $s_\mu$ $(\alpha = 1)$, and zonal 
polynomials $(\alpha = 1/2,\ \alpha = 2)$, each significant in their own right.  

Despite their relations to many well studied families of polynomials, Jack polynomials are comparatively poorly understood. 
One area that has seen some progress is their positivity in other bases.  From the definition of the $J_{\mu}^{(\alpha)}$, it is not obvious that the coefficients of the monomial expansion are in $\mathbb{Z}[\alpha]$, but this integrality conjecture was proven by Lapointe and Vinet  \cite{LaVi95b}.   A celebrated result of Knop and Sahi \cite{KnSa97} obtained later gives an explicit combinatorial formula for the expansion of $J_{\mu}^{(\alpha)}$ in the monomial basis, implying the stronger result that the coefficients lie in $\mathbb{N}[\alpha]$.  

Up until now there have not been any conjectures involving the expansion of $J_{\mu}^{(\alpha)}$  in the Schur basis;  the integrality result of 
Lapointe and Vinet implies these coefficients are in $\mathbb Z [\alpha]$, but computations show that they are
not generally in $\mathbb N [\alpha]$.    However, we have discovered that if we first define 
$$
{\tilde J}_{\mu}^{(\alpha)} (X) = \alpha ^n J_{\mu}^{(1/\alpha)} (X),
$$
then take the the coefficient of a given Schur function $s_{\lambda}(X)$ in ${\tilde J}_{\mu}^{(\alpha)} (X)$ 
and expand it either in the basis $\{ {\alpha + k \choose n } \}$ or in $\{{\alpha\choose k} k!\}$, the coefficients
appear to be nonnegative integers.  This opens up the intriguing question of whether or not these nonnegative integers have 
a combinatorial interpretation.   So far we been unable to find such an interpretation for general $\mu, \lambda$, but we hope this paper will inspire further research in this direction which will eventually lead to a solution to this question.

Our work on the $J_{\mu}^{(\alpha)}$ grew out of a conjecture about the (integral form, type A) Macdonald polynomials $J_{\mu}(X;q,t)$.
For $\mu$ a partition of $n$, it is well-known that
\begin{align*}
J_{\mu}^{(\alpha)}(X) = \lim _{t \to 1} J_{\mu}(X;t^{\alpha},t)/(1-t)^n,
\end{align*}
(see \cite[Chapter 6]{Macdonald} for background on Jack polynomials and Macdonald polynomials).  
The second author conjectures that for $k \in \mathbb N$ and $\mu, \lambda$ partitions of $n$,
\begin{align}
\label{Mconjecture}
\langle J_{\mu}(X;q,q^k)/(1-q)^n, s_{\lambda}(X) \rangle \in \mathbb N[q],
\end{align}
where $\langle\ ,\ \rangle$ is the usual Hall scalar product with respect to which the Schur functions are orthonormal, so Schur positivity seems to hold in a certain sense.   (It is a famous result of Mark Haiman \cite{Hai01} that the coefficients obtained by 
expanding $J_{\mu}(X;q,t)$ into the ``plethystic Schur" basis
$\{s_{\lambda}[X(1-t)]\}$ are in $\mathbb N[q,t]$, but little is known about expansions of $J_{\mu}(X;q,t)$ into the Schur basis. )  
The combinatorial formula of
Haglund, Haiman, and Loehr for the $J_{\mu}(X;q,t)$ \cite{HHL05b} implies that the coefficients obtained when expanding $J_{\mu}(X;q,q^k)/(1-q)^n$ into 
monomial symmetric functions are in $\mathbb N[q]$, and Yoo \cite{Yoo12, Yoo15} has proven that (\ref{Mconjecture}) holds
for certain important special cases, including when $\mu$ has two columns and when $\mu = (n)$.  
Beyond proving (\ref{Mconjecture}) 
in general, one could hope to be able to replace the nonnegative integer $k$ in the conjecture by a 
continuous parameter $\alpha$.  It is fairly easy to show that we have
\begin{align*}
{\tilde J}_{\mu}^{(\alpha)}(X) = \lim _{q \to 1} J_{\mu}(X;q,q^{\alpha})/(1-q)^n,
\end{align*}
which suggests that investigating the Schur expansion of $\alpha ^n \tilde{J}_\mu^{(\alpha)}(X)$ may shed some light on this case. In this article we introduce the 
following conjectures for this expansion, which have been tested using J. Stembridge's Maple package SF 
\cite{STE}.

\begin{conjecture}
\label{C1}
Let $\mu$, $\lambda$ be partitions of $n$. Then setting
$$\langle \tilde{J}_\mu^{(\alpha)}(X), s_\lambda \rangle = \sum_{k=0}^{n-1}  a_{k}(\mu,\lambda){\alpha + k \choose n},$$
we have $a_{k}(\mu,\lambda) \in \mathbb{N}$.  Furthermore, the polynomial
$\sum_{k=0}^n a_k(\mu,\lambda) z^k$ has only real zeros.
\end{conjecture}

\begin{conjecture}
\label{C2}
Let $\mu$, $\lambda$ be partitions of $n$. Then setting
$$\langle \tilde{J}_\mu^{(\alpha)}(X), s_\lambda \rangle = \sum_{k=1}^n b_{n-k}(\mu,\lambda) {\alpha \choose k} k! ,$$
we have $b_{n-k}(\mu,\lambda) \in \mathbb{N}$. Furthermore, the polynomial $\sum_{k=0}^n b_{n-k}(\mu,\lambda)z^k$ has only real zeros.
\end{conjecture}

It turns out that part one of Conjecture \ref{C1} (almost) implies Conjecture \ref{C2}.  The identity
$
{\alpha + k  \choose n} = \sum _{i} {\alpha \choose i} {k \choose n-i}
$
shows that if the $a_{k}(\mu,\lambda) \in \mathbb N$, then $k! b_{n-k}(\mu,\lambda) \in \mathbb N$, so if Conjecture \ref{C1} is true, the only issue is
whether or not the $b_{n-k}(\mu,\lambda)$ are integers.

Using Yoo's results on (\ref{Mconjecture}) and other methods, we prove various special cases of Conjectures $1$ and $2$, which suggest that there are rich combinatorics lurking in these expansions. In particular, we find that both Eulerian and Stirling numbers appear, as well as rook boards and generating functions of quasi-Yamanouchi tableaux. Our results also point towards attractive conjectures for the fundamental quasisymmetric expansion in each basis.

\section{Preliminaries}

A \textit{partition} $\mu = (\mu_1 \geq \mu_2 \geq \cdots \geq \mu_k > 0)$ is a finite sequence of non-increasing positive integers. The \textit{size} of $\mu$, denoted $|\mu|$, is the sum of the integers in the sequence. The \textit{length} of $\mu$ is the number of integers in the partition, denoted $\ell(\mu)$. We say that $\mu$ \textit{dominates} $\lambda$ if $\mu_1 + \cdots + \mu_i \geq \lambda_1 + \cdots + \lambda_i$ for all $i \geq 1$. If $\mu$ dominates $\lambda$, then we write $\mu \geq \lambda$, forming a partial order on partitions.
We use French notation to draw the \textit{diagram} corresponding to a partition $\mu$ by having left justified rows of boxes starting at the bottom, where the $i$th row has $\mu_i$ boxes. The conjugate of $\mu$ is obtained by taking the diagram of $\mu$ and reflecting across the diagonal.

Bases of the ring of symmetric functions are indexed by partitions, and in particular we write $e_\mu$, $h_\mu$, $m_\mu$, and $p_\mu$ to denote the elementary, complete homogeneous, monomial, and power-sum symmetric functions respectively and $s_\mu$ to denote the Schur functions. We also write $Q_\sigma(x)$ for the fundamental quasisymmetric function, where $\sigma \subseteq \{1,\ldots, n-1\}$ and 
$Q_\sigma(x) = \sum_{\substack{i_1 \leq \cdots \leq i_n \\ j\in \sigma \Rightarrow i_j < i_{j+1}}} x_{i_1}\cdots x_{i_n}$.





\subsection{Eulerian and Stirling numbers}

A permutation $\pi = \pi_1 \pi_2 \cdots \pi_n \in S_n$ has a descent at position $i$ if $\pi_i > \pi_{i+1}$. Write $A(n,k)$ for the Eulerian number counting permutations in $S_n$ with $k$ descents and $S(n,k)$ for the Stirling number of the second kind counting the ways to partition $n$ labelled objects into $k$ nonempty, unlabelled subsets. 

For $|\lambda| = n$, we define the set of $\lambda$-restricted permutations to be permutations where $1, 2,\ldots, \lambda_1$ appear in order, $\lambda_1+1, \ldots, \lambda_1 +\lambda_2$ appear in order, etc. We define the $\lambda$-restricted Eulerian number $A(\lambda,k)$ to be the the number of $\lambda$-restricted permutations in $S_n$ with $k$ descents.
\begin{example}
The $10$ $(3,2)$-restricted permutations are $12345$, $12435$, $14235$, $41235$, $12453$, $14253$, $41253$, $14523$, $41523$, and $45123$.
\end{example}

\subsection{Quasi-Yamanouchi Tableaux}
A \textit{semistandard Young tableau} $T$ is a filling of the diagram of a partition $\mu$ using positive integers that weakly increase to the right and strictly increase upwards. We say that $T$ has shape $\mu$ and write SSYT$_m(\mu)$ to denote the set of semistandard Young tableaux of shape $\mu$ with maximum value at most $m$.

If $w_i = w_i(T)$ is the number of entries of $T$ with value $i$, then we say that $T$ has \textit{weight} $w = (w_1, w_2, \ldots)$. The Kostka number $K_{\mu\lambda}$ counts the number of semistandard Young tableaux of shape $\mu$ and weight $\lambda$. 
Let $|\mu| = n$. A \textit{standard Young tableau} is a semistandard Young tableau with weight $(1^n)$. The set of standard Young tableaux of shape $\mu$ is SYT$(\mu)$.

We say that an entry $i$ is weakly left of an entry $j$ in a tableau $T$ when $i$ is above or above and to the left of $j$. The \textit{descent set} of a tableaux $T$ is the set of entries $i \in Des(T) \subseteq \{1,\ldots,n-1\}$ such that $i+1$ is weakly left of $i$, and we write $des(T) = |Des(T)|$. Given a descent set $Des(T) = \{d_1, d_2,\ldots,d_{k-1}\}$, the $i$th run of $T$ is the set of entries from $d_{i-1}+1$ to $d_{i}$, where $1 \leq i \leq k$, $d_0 = 0$, and $d_k = n$.

\begin{example} A standard Young tableau with five runs. The first and fourth runs are highlighted, and the tableau has descent set $\{3,6,8,11\}$.
\begin{displaymath}
	\tableau{
	\mathclap{\graybox}\mathclap{\raisebox{4\unitlength}{9}}&\mathclap{\graybox}\mathclap{\raisebox{4\unitlength}{10}}& 12\\
	4&5&7&\mathclap{\graybox}\mathclap{\raisebox{4\unitlength}{11}}\\
	\mathclap{\grayboxtwo}\mathclap{\raisebox{4\unitlength}{1}}&\mathclap{\grayboxtwo}\mathclap{\raisebox{4\unitlength}{2}}&\mathclap{\grayboxtwo}\mathclap{\raisebox{4\unitlength}{3}} & 6&8\\
	}
	\end{displaymath}
\end{example}

A semistandard Young tableau is \textit{quasi-Yamanouchi} if when $i$ appears in the tableau, the leftmost instance of $i$ is weakly left of some $i-1$. QYT$_{\leq m}(\mu)$ is the set of quasi-Yamanouchi tableaux with maximum value at most $m$, and QYT$_{=m}(\mu)$ is the set of quasi-Yamanouchi tableaux with maximum value exactly $m$. 

\begin{example} All the quasi-Yamanouchi tableaux of shape $(2,2,1)$, demonstrating that $\textnormal{QYT}_{=3}(2,2,1)=3$ and $\textnormal{QYT}_{=4}(2,2,1) = 2$.
\begin{displaymath}
	\tableau{
	3\\
	2&2\\
	1&1\\}
	\ \ \ \ \ \tableau{
	3\\
	2&3\\
	1&1\\}
	\ \ \ \ \ \tableau{
	3\\
	2&3\\
	1&2\\}
	\ \ \ \ \ \tableau{
	4\\
	2&3\\
	1&2\\}
	\ \ \ \ \ \tableau{
	3\\
	2&4\\
	1&3\\}
	\end{displaymath}
\end{example}

Quasi-Yamanouchi tableaux first arose as objects of interest in the work of Assaf and Searles \cite{AsSe17}, where they showed that quasi-Yamanouchi tableaux could be used to tighten Gessel's expansion of Schur polynomials into fundamental quasisymmetric polynomials. Some of their combinatorial properties and a partial enumeration can be found in \cite{Wa16}.

A property of particular interest is that quasi-Yamanouchi tableaux of shape $\mu$ have a natural bijection with standard Young tableaux of shape $\mu$ via the destandardization map that Assaf and Searles use \cite[Definition 2.5]{AsSe17}.
Given a semistandard Young tableau $T$, choose $i$ such that the leftmost $i$ is strictly right of the rightmost $i-1$, meaning to the right or below and to the right, or such that there are no $i-1$, and decrement every $i$ to $i-1$. Repeat until no more entries $i$ can be decremented. The resulting tableau is the \textit{destandardization} of $T$, and applying the inverse gives the standardization. We get the following result using this map, which appears in  \cite{AsSe17} and \cite{Wa16}.

\begin{proposition}
Let $\mu = (\mu_1, \mu_2, \ldots, \mu_k)$ be a partition of size $n$, then
$\textnormal{QYT}_{\leq n} (\mu) \cong \textnormal{SYT}(\mu).$
\end{proposition}

Note that by definition, QYT$_{=m}(\mu)=0$ for any $|\mu| = n$ and $m>n$, so QYT$_{\leq n}(\mu)$ contains all quasi-Yamanouchi tableaux of shape $\mu$. Since we can partition QYT$_{\leq n}(\mu)$ into $\{\textrm{QYT}_{=m}(\mu)\  |\ 1\leq m \leq n\}$, this bijection gives a refinement on standard Young tableaux. In particular, the image of QYT$_{=m}(\mu)$ is the subset of SYT$(\mu)$ with exactly $m$ runs.
We will also use the following result, which is obtained through a bijection consisting of standardizing a quasi-Yamanouchi tableau, conjugating, and then destandardizing.
\begin{lemma}\label{QYTsymmetry}
Given a partition $\lambda$ of $n$, its conjugate $\lambda'$, and $1\leq k \leq n$, $\textnormal{QYT}_{=k}(\lambda) = \textnormal{QYT}_{=(n+1)-k}(\lambda')$.
\end{lemma}

\subsection{Dual Equivalence}
We will need Assaf's dual equivalence graphs \cite{As15a}, although not in their full generality.
Define the \textit{elementary dual equivalence involution} $d_i$ on $\pi \in S_n$ for $1<i<n$ by $d_i(\pi) = \pi$ if $i-1$, $i$, and $i+1$ appear in order in $\pi$ and by $d_i(\pi) = \pi'$ where $\pi'$ is $\pi$ with the positions of $i$ and whichever of $i\pm1$ is further from $i$ interchanged when they do not appear in order. Two permutations $\pi$ and $\tau$ are \textit{dual equivalent} when $d_{i_1}\cdots d_{i_k}(\pi) = \tau$ for some $i_1,\ldots,i_k$. The reading word of a tableau is obtained by reading the entries from left to right, top to bottom, which for standard Young tableaux produces a permutation, and two standard Young tableaux of the same shape are dual equivalent if their reading words are.

We also use Assaf's characterization of Gessel's expansion of the Schur function into the fundamental quasisymmetric basis
$s_\mu = \sum_{T\in [T_\mu]} Q_{Des(T)}(x)$
where $[T_\mu]$ is the dual equivalence class of all standard Young tableaux of shape $\mu$, which is in fact all standard Young tableaux of shape $\mu$.

\subsection{Rook Boards}



We will use the following result from Goldman, Joichi, and White \cite{GJW75} which translates certain products of factors into each of our bases.

\begin{proposition}\label{rookboardprop}
Let $0\leq c_1\leq c_2\leq \cdots \leq c_n \leq n$ with $c_i \in \mathbb{N}$, and let $B = B(c_1,\ldots, c_n)$ be the Ferrers board whose $i$th column has height $c_i$. Then
$$\prod_{i=1}^n (\alpha + c_i -i + 1) = \sum_{k=0}^n h_k(B){\alpha +k\choose n} = \sum_{k=0}^n r_{n-k}(B) {\alpha \choose k} k!,$$
where $h_k(B)$ is the number of ways of placing $n$ nonattacking rooks on the $n\times n$ grid containing $B$ with exactly $k$ rooks on $B$ and $r_k(B)$ is the number of ways of placing $k$ nonattacking rooks on $B$.
\end{proposition}


\section{Partial Results} 
\subsection{Eulerian and Stirling numbers}
We first became interested in Conjecture 1 and Conjecture 2 due to the following observation, which follows from the normalization property of Jacks and the way that $\alpha^n$ is written in these two bases.
\begin{proposition} \label{eulerianstirlingprop}
For a partition $\mu$, the coefficient of $m_{1^n}$ in $\tilde{J}_\mu^{(\alpha)}(X)$ is
$$\langle \tilde{J}_\mu^{(\alpha)}(X), h_{1^n}\rangle = n!\alpha^n = \sum_{k=0}^{n-1} n! A(n,k) {\alpha + k \choose n} = \sum_{k=1}^{n} n! S(n,k){\alpha \choose k} k!.$$
\end{proposition}
Combined with computer data confirming the two conjectures up to $n=11$, we have the first hint that these bases may have interesting combinatorial interpretations. Two immediate corollaries come from extracting the coefficient of $m_{1^n}$ from $s_\lambda$ in the Schur expansion of $\tilde{J}_\mu^{(\alpha)}(X)$.
\begin{corollary} \label{eulerianstirlingcor1}
Given a partition $\mu$, we have
$$\sum_{|\lambda| = n} \sum_{k=0}^{n-1} a_{k}(\mu,\lambda) K_{\lambda(1^n)} {\alpha + k \choose n} = \sum_{k=0}^{n-1} n!A(n,k){\alpha + k \choose n}.$$
\end{corollary}
\begin{corollary}
Given a partition $\mu$, we have
$$\sum_{|\lambda| = n} \sum_{k=1}^{n-1} b_{n-k}(\mu,\lambda) K_{\lambda(1^n)} {\alpha \choose k} k! = \sum_{k=1}^{n-1} n!S(n,k){\alpha \choose k} k!.$$
\end{corollary}
Furthermore, if $a_k(\mu,\lambda) \in \mathbb{N}[\alpha]$ or $b_k(\mu,\lambda) \in \mathbb{N}[\alpha]$ in general, then the respective result above would indicate some refinement on Eulerian numbers or Stirling numbers of the second kind.

\subsection{Quasi-Yamanouchi Tableaux}

In the case of $\mu = (n)$ and $|\lambda| = n$, we noticed that the equality
$$\sum_{k=0}^{n-1} a_k((n),\lambda) = |\textrm{SYT}(\lambda)|$$
held for the computer generated data. Upon closer inspection, it appeared that in fact the following theorem was true.

\begin{theorem}\label{QYTtheorem}
Let $\lambda$ be a partition of $n$ and $\lambda'$ be its conjugate. Then for the coefficient of $s_\lambda$ in $\tilde{J}_{(n)}^{(\alpha)}(X)$
$$\langle \tilde{J}_{(n)}^{(\alpha)}(X), s_\lambda\rangle = \sum_{k=0}^{n-1} a_k((n),\lambda) {\alpha + k \choose n},$$
we have $a_k((n),\lambda) =n!  \textnormal{QYT}_{=k+1}(\lambda')$.
\end{theorem}

We split the proof into several parts, starting with the coefficient of $m_\lambda$ in $J_{(n)}^{(\alpha)}(X)$. By example 3 in chapter VI, section 10 of Macdonald \cite{Macdonald}, this is $\frac{n!}{\lambda!}\prod_{s\in \lambda} (arm(s)\alpha +1)$ where $\lambda! = \lambda_1!\lambda_2!\cdots$ and the arm of a cell $s$ in the diagram of $\lambda$ is the number of cells to the right of $s$. Converting to $\tilde{J}_{(n)}^{(\alpha)}(X)$, this becomes $\frac{n!}{\lambda!}\prod_{s\in \lambda} (\alpha + arm(s))$. The next step is to convert these coefficients to the new basis.

\begin{lemma}
Given a partition $\lambda$ of $n$,
$$\frac{n!}{\lambda!}\prod_{s\in \lambda} (\alpha + arm(s)) = n! \sum_{k=0}^{n-1} A(\lambda, k){\alpha + n - 1 - k \choose n},$$
where $A(\lambda,k)$ is the number of $\lambda$-restricted permutations with $k$ descents. 
\end{lemma}

\begin{proofsketch}
Cancel $n!$ and rewrite the left hand side to get 
$$\prod_{i=1}^{\ell(\lambda)} {\alpha + \lambda_i -1\choose \lambda_i} =\sum_{k=0}^{n-1} A(\lambda, k){\alpha + n - 1 - k \choose n},$$ 
then show these are equal through a bijective algorithm.
\end{proofsketch}

We now wish to relate these $\lambda$-restricted Eulerian numbers to quasi-Yamanouchi tableaux. We can achieve this through the Robinson-Schensted-Knuth (RSK) correspondence. 
\begin{lemma} \label{RSKlemma}
Given a partition $\lambda$ of $n$, it holds that
$$\sum_{k=0}^{n-1} A(\lambda,k){\alpha +n -1-k \choose n} = \sum_{\substack{|\nu| = n \\ \nu \geq \lambda}} K_{\nu\lambda} \sum_{k=0}^{n-1} \textnormal{QYT}_{=k+1}(\nu){\alpha + n -1-k \choose n}.$$
\end{lemma}
\begin{proofsketch}
By comparing coefficients of ${\alpha + n -1 -k\choose n}$, it is sufficient to show that for a fixed $k$,
$$A(\lambda, k) = \sum_{\substack{|\nu|=n \\ \nu \geq \lambda}}K_{\nu\lambda}\textnormal{QYT}_{=k+1}(\nu).$$
This is proven bijectively with RSK.
\end{proofsketch}

\begin{corollary}
It holds that
$$\sum_{|\lambda| = n} \sum_{k=0}^{n-1} A(\lambda, k){\alpha +n-1-k\choose n} m_\lambda = \sum_{|\nu|=n}\sum_{k=0}^{n-1} \textnormal{QYT}_{=k+1}(\nu){\alpha+n-1-k\choose n} s_\nu.$$
\end{corollary}
\begin{proofsketch}
Apply induction and Lemma \ref{RSKlemma} to the poset of partitions under dominance order.
\end{proofsketch}

Linking these together and applying Lemma \ref{QYTsymmetry} completes the proof of Theorem \ref{QYTtheorem}, which shows that the Schur expansion of $\tilde{J}_\mu^{(\alpha)}(X)$ is in fact a generating function for quasi-Yamanouchi tableaux up to a constant of $n!$, thus proving Conjecture 1 for the case of $\mu = (n)$. Chen, Yang, and Zhang \cite{CYZ16} adapted a result by Brenti \cite{Br89} to show that the polynomial
$\sum_{T\in \textnormal{SYT}(\lambda)} t^{des(T)}$
has only real zeroes. Using this and the definition of quasi-Yamanouchi tableaux, Theorem \ref{QYTtheorem} also proves the $\mu = (n)$ case of the second part of Conjecture 1.

\subsection{Fundamental quasisymmetric expansions}
One last application of Theorem \ref{QYTtheorem} takes us into a brief digression towards the fundamental quasisymmetric expansion. In the ${\alpha + k \choose n}$ basis, we can obtain the fundamental quasisymmetric expansion of $\tilde{J}_{(n)}^{(\alpha)}(X)$ as a corollary of the following result.

\begin{theorem}\label{QYTqsymexp} It holds that
$$\sum_{\pi \in S_n} t^{des(\pi)} Q_{Des(P(\pi))}(x)= \sum_{|\mu|=n} \sum_{k=0}^{n-1}\textnormal{QYT}_{=k+1}(\mu)t^k s_\mu,$$
where $P(\pi)$ is the insertion tableau of $\pi$ given by RSK.
\end{theorem}

\begin{proofsketch}
Connect elements of $S_n$ by edges corresponding to dual equivalence relations, then apply RSK to each element and use the theory of dual equivalence graphs.
\end{proofsketch}

\begin{corollary} \label{jacksqsymmexp}
It holds that
$$\tilde{J}_{(n)}^{(\alpha)}(X) = n!\sum_{\pi \in S_n} {\alpha +n-1-des(\pi)\choose n} Q_{Des(P(\pi))},$$
where $P(\pi)$ is the insertion tableau of $\pi$ given by RSK.
\end{corollary}
\begin{proof}
Apply Theorems \ref{QYTtheorem} and \ref{QYTqsymexp} and Lemma \ref{QYTsymmetry}.
\end{proof}

This result prompted the following conjecture on the quasisymmetric expansion for general partitions $\mu$. 

\begin{conjecture}
For a partition $\mu$ of size $n$, it holds that
$$\tilde{J}_\mu^{(\alpha)}(X) = \sum_{\pi,\tau \in S_n} {\alpha + n -1 - des(\pi) \choose n} Q_{\sigma(\pi,\tau,\mu)}(x)$$
for some set-valued function $\sigma$ depending on $\pi, \tau,$ and $\mu$ and with image in $\{1,\ldots,n-1\}$.
\end{conjecture}

Corollary \ref{jacksqsymmexp} proves this conjecture in the case of $\mu=(n)$, where $\sigma(\pi,\tau,(n)) = Des(P(\pi))$, and Proposition \ref{eulerianstirlingprop} proves it in the case of $\mu = (1^n)$, where $\sigma(\pi,\tau,(1^n)) = \{1,\ldots,n-1\}$ for all $\pi,\tau \in S_n$. Furthermore, if we momentarily assume that the Jack polynomials are Schur positive in this basis, Corollary \ref{eulerianstirlingcor1} along with the expansion of Schurs into fundamental quasisymmetrics shows that this conjecture is true in general for some $\sigma$, although it does not tell us what $\sigma$ should be.
Finally, while the fundamental quasisymmetric expansion would be interesting in its own right, it may also lead to a proof of Schur positivity by a generalization of the method used in Theorem \ref{QYTqsymexp}. Corollary \ref{jacksqsymmexp} and Conjecture 5 have the following analogous conjectures in the ${\alpha \choose k} k!$ basis, where $B_n$ is the set of set partitions of $\{1,\ldots,n\}$.

\begin{conjecture}\label{conjbasis2}
For a partition $\mu$ of size $n$, it holds that
$$\tilde{J}_\mu^{(\alpha)}(X) = \sum_{\substack{\pi \in S_n \\ \beta \in B_n}} {\alpha \choose |\beta|}|\beta|! Q_{\rho(\pi,\beta,\mu)}(x)$$
for some set-valued function $\rho$ depending on $\pi, \beta$, and $\mu$ with image in $\{1,\ldots,n-1\}$.
\end{conjecture}

For the $\mu=(n)$ case, we first define a function. Given $\pi \in S_n$ and $\beta \in B_n$, define $f_\beta(\pi)$ to be a rearrangement of $\pi$ so that if $\{b_1,\ldots,b_k\}\in \beta$, then $b_1,\ldots,b_k$ appear in increasing order in $f_\beta(\pi)$ without changing the position of the subsequence. For example, given $\beta = \{\{1,4\},\{2,3,5\}\}$ and $\pi = 24531$, $f_\beta(\pi) = 21354$.

\begin{conjecture}\label{conjbasis2muequalsn}
It holds that
$$\tilde{J}_{(n)}^{(\alpha)}(X) = \sum_{\substack{\pi \in S_n \\ \beta \in B_n}} {\alpha \choose |\beta|}|\beta|! Q_{Des(P(f_\beta(\pi)))}(x),$$
where $P(f_\beta(\pi))$ is the insertion tableau of $f_\beta(\pi)$ given by RSK.
\end{conjecture}

Proposition \ref{eulerianstirlingprop} also proves the $\mu = (1^n)$ case here, and we can make similar remarks as above. That is, if we assume Schur positivity, that Conjecture \ref{conjbasis2} is true for some $\rho$ and that the quasisymmetric expansion could help prove Schur positivity in this basis.

\subsection{Rook Boards}
Returning to the problem of Schur positivity, we also found some success approaching the problem with rook boards. We first use Proposition \ref{rookboardprop} to obtain a combinatorial interpretation of the coefficient of $s_\lambda$ in our binomial bases when $\mu = \lambda$ for a hook shape. In general for $\mu = \lambda$, the coefficient of $s_\mu$ is the same as the coefficient of $m_\mu$ in the monomial expansion, so we can obtain from the combinatorial formula for the monomial expansion \cite{KnSa97} that
$$\langle \tilde{J}_\mu^{(\alpha)}(X), s_{\mu}(X)\rangle = \prod_{s\in \mu} (arm(s) + \alpha (leg(s) +1)).$$
When $\mu = (n-\ell, 1^\ell)$ is a hook shape, this product becomes
\begin{align*}
\langle \tilde{J}_\mu^{(\alpha)}(X),&\ s_{\mu}(X)\rangle = \ell!\alpha^\ell ((\ell+1)\alpha + (n-1))(\alpha + (n-2))\cdots(\alpha + 1) \alpha\\
=& \ell\cdot \ell!(\alpha + (n-\ell-2))\cdots(\alpha + 1)\alpha^{\ell+2}+ \ell!(\alpha+(n-\ell-1))\cdots(\alpha + 1)\alpha^{\ell+1},
\end{align*}
then applying Proposition \ref{rookboardprop} gives the following result.

\begin{proposition} For $\mu = \lambda = (n-\ell, 1^\ell)$, we have
\begin{align*}
\langle \tilde{J}_\mu^{(\alpha)}(X),s_\mu \rangle &= \sum_{k=0}^n (\ell \cdot \ell! h_k(B(c_1,\ldots,c_n)) + \ell! h_k(B(d_1,\ldots,d_n))) {\alpha + k \choose n }\\
& = \sum_{k=0}^n (\ell \cdot \ell! r_k(B(c_1,\ldots,c_n)) + \ell! r_k(B(d_1,\ldots,d_n))) {\alpha \choose k }k!\\
\end{align*}
where $c_1 = c_2 = \cdots = c_{n-\ell-1} = n-\ell-2$ and $c_{n-\ell+i} = n-\ell-1+i$ for $0 \geq i \geq \ell$ and $d_1 = d_2 = \cdots = d_{n-\ell} = n-\ell-1$ and $d_{n-\ell+i} = n-\ell-1+i$ for $1 \geq i \geq \ell$.
\end{proposition}

\begin{example}$B(c_1,\ldots,c_4)$ and $B(d_1,\ldots,d_4)$ for $\mu=\lambda=(3,1)$.
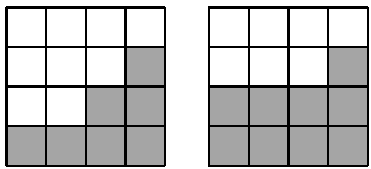
\begin{displaymath}
\tableau{
\ & \ &\ &\ \\
\ & \ & \ & \graybox\\
\ & \ & \graybox & \graybox \\
\graybox & \graybox & \graybox & \graybox \\}
\ \ \ \ \ 
\tableau{
\ & \ &\ &\ \\
\ & \ & \ & \graybox\\
\graybox & \graybox & \graybox & \graybox \\
\graybox & \graybox & \graybox & \graybox \\}
\end{displaymath}
\end{example}

This approach also yields a combinatorial interpretation for both bases in the case of $\mu = (n)$. We can obtain $J_{\mu}^{(1/\alpha)}(X)$ via the specialization\\

\noindent\resizebox{1.0\textwidth}{!}{
$J_{\mu}^{(1/\alpha)}(X) = \lim_{t\rightarrow 1} \frac{J_{\mu}(X;t^{1/\alpha},t)}{(1-t)^n} = \lim_{q\rightarrow 1} \frac{J_{\mu}(X;q,q^\alpha)}{(1-q^\alpha)^n}=\lim_{q\to 1} \frac{J_{\mu}(X;q,q^\alpha)}{(1-q)^n}\frac{(1-q)^n}{(1-q^\alpha)^n}=\frac{1}{\alpha^n}\lim_{q \to 1} \frac{J_{\mu}(X;q,q^\alpha)}{(1-q)^n}$}
\\

\noindent so that 
$$\tilde{J}_\mu^{(\alpha)}(X) = \lim_{q\to 1} \frac{J_{\mu}(X;q,q^\alpha)}{(1-q)^n}.$$
Then when $\mu = (n)$, we can apply the limit as $q \to 1$ 
to a result of Yoo \cite[Theorem 3.2]{Yoo12} to obtain
$$\lim_{q\to 1}\frac{J_{(n)}(X;q,q^\alpha)}{(1-q)^n} = \sum_{|\lambda|=n}s_{\lambda} K_{\lambda,1^n} \prod_{(i,j)\in \lambda}(\alpha+i-j),$$
where $(i,j)\in \lambda$ refers to cells of the diagram of $\lambda$ identified with their Cartesian coordinates. Arrange the values of $i-j$ in non-increasing order and rewrite to get the desired $\prod_{i=1}^n (\alpha + c_i-i+1)$ form. It is clear that for any partition $\lambda$, this produces a sequence $0\leq c_1 \leq \cdots \leq c_n \leq n$, so we can apply Proposition \ref{rookboardprop} again to obtain the following.

\begin{theorem}\label{rookboardtheorem}
It holds that 
$$\langle \tilde{J}_{(n)}^{(\alpha)}(X),s_\lambda\rangle = \sum_{k=0}^n K_{\lambda,1^n} h_k(B(c_1,\ldots,c_n)) {\alpha + k \choose n} = \sum_{k=0}^n K_{\lambda,1^n} r_{n-k}(B(c_1,\ldots,c_n)) {\alpha \choose k}k!,$$
with $c_1,\ldots,c_n$ given above.
\end{theorem}
\begin{example} The values $i-j$ and $B(c_1,\ldots,c_5)$ for $\lambda = (3,2)$.
 \begin{displaymath}
    \raisebox{-45pt}{\tableau{
    -1 & 0 \\
    0 & 1 & 2 \\
    }}\ \ \ \ \ 
    \tableau{
    \ & \ & \ & \ & \ \\
    \ & \ & \ & \ & \ \\
    \ & \ & \ & \graybox & \graybox \\
    \graybox & \graybox & \graybox & \graybox & \graybox \\
    \graybox & \graybox & \graybox & \graybox & \graybox \\}
    \end{displaymath}
\end{example}

In \cite{HOW} it is shown that the rook and hit polynomials of Ferrers boards have only real zeros, so the two results of this section also prove the second part of Conjecture \ref{C1}  for these special cases. 
We note that Theorem \ref{rookboardtheorem} provides a very different looking combinatorial interpretation to the one seen in Theorem \ref{QYTtheorem} for the ${\alpha + k \choose n}$ basis. One path that we plan to pursue next is to find a bijection between the rook board interpretation of Theorem \ref{rookboardtheorem} and the tableau interpretation of Theorem \ref{QYTtheorem}, which we hope will lead to a tableau interpretation for the ${\alpha\choose k }k!$ basis.


\printbibliography

\end{document}